\documentclass
[
11pt]
{amsart}%

\usepackage{hyperref}
\usepackage{cite}
\usepackage{amssymb}
\usepackage{amsmath}
\usepackage{amsthm}
\usepackage{amsfonts}
\usepackage{graphicx}
\usepackage{bbm}
\usepackage{xcolor}
\usepackage[toc,page]{appendix}
\usepackage{subfigure}
\usepackage{mathtools}
\usepackage{bm} 
\usepackage[normalem]{ulem}
\usepackage{comment}

\newtheorem{theorem}{Theorem}[section]

\newtheorem{proposition}[theorem]{Proposition}

\newtheorem{lemma}[theorem]{Lemma}
\newtheorem{remark}[theorem]{Remark}
\newtheorem{example}[theorem]{Example}

\newtheorem*{acknowledgement}{Acknowledgement}
\newtheorem*{theorem*}{Theorem}

\theoremstyle{plain}
\theoremstyle{definition}
\newtheorem{definition}{Definition}[section]
\newtheorem{notation}{Notation}[section]

\numberwithin{equation}{section}

\newcommand{\R}{\mathbb{R}}  
\newcommand{\E}{\mathbb{E}} 
\newcommand{\Prob}{\mathbb{P}}
\newcommand{\Q}{\mathbb{Q}}
\newcommand{\Hei}{\mathbb{H}}

\newcommand{\WR}{W_0 \left(\mathbb{R}^2\right)}
\newcommand{\WH}{W_0 \left( \mathbb{H} \right)}

\DeclareMathOperator{\Span}{span}

\begin{document}


\title[Onsager-Machlup on Heisenberg group]{On the Onsager-Machlup functional for the Brownian motion on the Heisenberg group}

\author[Carfagnini]{Marco Carfagnini{$^{\dag }$}}
\address{ Department of Mathematics\\
University of California, San Diego\\
La Jolla, CA 92093-0112,  U.S.A.}
\email{mcarfagnini@ucsd.edu}

\author[Gordina]{Maria Gordina{$^{\dag }$}}
\thanks{\footnotemark {$\dag$} Research was supported in part by NSF Grant DMS-1954264.}
\address{ Department of Mathematics\\
University of Connecticut\\
Storrs, CT 06269,  U.S.A.}
\email{maria.gordina@uconn.edu}

\keywords{Diffusion processes, Onsager-Machlup functional, Heisenberg group, hypoelliptic operator}

\subjclass[2010]{Primary 58J65; Secondary  60J60, 60G17, 35R03, 53C17}

\begin{abstract}
Onsager-Machlup functionals are used to describe the dynamics of a continuous stochastic process. For a Markov process whose generator is an elliptic operator, such functionals have been studied extensively. We describe the Onsager-Machlup functional for a horizontal Brownian motion on a Heisenberg group. This is a Markov process whose generator is a hypoelliptic differential operator. Unlike in the elliptic case we do not rely on  tools from differential geometry such as comparison theorems or curvature bounds as these are not easily available in the sub-Riemannian setting. In addition, we study fine properties of trajectories of the horizontal Brownian motion, including a new notion of horizontal continuous curves.
\end{abstract}

\maketitle

\tableofcontents

\section{Introduction}\label{Intro}

In this article we study small fluctuations of stochastic processes in terms of Onsager-Machlup functionals. These functionals are used to describe asymptotic behavior of the probability that the paths of a given stochastic process are contained in a small tube around a given curve. The purpose of this paper is to find an Onsager-Machlup functional and corresponding asymptotics in a hypoelliptic setting.

Onsager-Machlup functionals were introduced in  \cite{MachlupOnsager1953a, MachlupOnsager1953b} to determine the most probable path of a diffusion process and can be considered as a probabilistic analogue of the Lagrangian of a dynamical system. Onsager's principle has been used in non-equilibrium statistical mechanics and thermodynamics where connections to the large deviations principle have been observed in physics literature. For example, in \cite{Oono1993} Y.~Oono noted that if one interprets the macroscopic time derivatives as the time average of microscopic time derivatives, then  the rate function for the corresponding  large deviations principle is the Onsager-Machlup Lagrangian, thus leading to a general principle that the rate function of microscopic dynamics governs the phenomenological time evolution of macroscopic quantities. This conclusion follows from using Gaussian approximations to the dynamics in question. For a more mathematical approach to microscopic versus macroscopic, large deviations and the Onsager-Machlup functionals we refer to \cite{Renger2022}.

For example, let $B_{t}$ be a real-valued Brownian motion on a probability space $\left( \Omega, \mathcal{F}, \Prob \right)$. Suppose one can find a functional $\mathcal{L}$ and a constant $C(\varepsilon) >0$ only depending on $\varepsilon$ such that for any $\varphi$ in $\mathcal{C}^2 ([0,1],\R)$
\begin{equation}\label{eqn.goal}
\Prob \left( \sup_{0 \leqslant t \leqslant 1} \lvert B_{t} - \varphi\left( t \right)\rvert  < \varepsilon  \right) \asymp C(\varepsilon)\exp\left( \mathcal{L}(\varphi) \right),  \text{ as }  \varepsilon \longrightarrow 0.
\end{equation}
Then $\mathcal{L}$ is called an \emph{Onsager-Machlup functional} for $B_{t}$ with respect to the sup-norm. In particular,  minimizing $\mathcal{L}$ would yield the most probable deterministic path for $B_{t}$. For finite energy paths $\exp\left( \mathcal{L}(\varphi) \right)$ is the corresponding Girsanov density.

One can consider  \eqref{eqn.goal} for different classes of stochastic processes, or smoothness of the  curve $\varphi$, and finally for tubes around the trajectories defined by different norms. Various norms on the path space and less regular curves have been considered as well, but not in a hypoelliptic setting. The literature on the subject is vast, and here we mention only the ones most relevant to the techniques we use in this paper. We refer to \cite{Zeitouni1989a} for the case of curves which are not necessarily~$C^2$,  and to \cite{LyonsZeitouni1999, SheppZeitouni1993}, where  $L^p$ and convex norms are considered. The case of a diffusion process on a Riemannian manifold with an elliptic infinitesimal generator is the subject of  \cite{Capitaine2000, FujitaKotani1982, HaraTakahashiArxiv2016, HaraTakahashi1996, TakahashiWatanabe1980}. One of the main tools in these results is Girsanov's theorem which is well illustrated by the exposition in \cite{DurrBach1978}. This allows a comparison to a diffusion with constant coefficients which is not applicable in our setting.

Onsager-Machlup functionals for anticipating processes have been studied in \cite{Chaleyat-MaurelNualart1992}, and the fractional Brownian motion is considered in \cite{MoretNualart2002}. We refer to \cite{Mayer-WolfZeitouni1993, DemboZeitouni1991} for the solution to some elliptic stochastic PDE on $\R^{d}$. Onsager-Machlup functionals for jump-diffusion processes are the content of  \cite{ChaoDuan2019}.  As one can see from \cite{GyongyNualartSanz-Sole1995, Chaleyat-MaurelNualart1995, CarmonaNualart1992, BardinaRoviraTindel2003a, BardinaRoviraTindel2003b}, the subject has been active and developed in many directions. Relevant to our setting is their reliance on Girsanov's theorem.

A hypoelliptic diffusion $X_{t}$ has been considered by P.~Pigato in \cite{Pigato2018}, where he provides upper and lower bounds of the probability that the paths of $X_{t}$ are contained in a tube around a given path. However, since the tubes in his approach cannot be arbitrarily small, these estimates are not sufficient to find an Onsager-Machlup functional for $X_{t}$.

The current paper describes Onsager-Machlup functionals and the corresponding asymptotics for a diffusion whose infinitesimal generator is not elliptic, thus making use of Girsanov's theorem either not possible or at least not straightforward. This case has not been considered so far, and we are starting with a particular case of a hypoelliptic diffusion, namely, a hypoelliptic Brownian motion on the Heisenberg group. This group has a natural structure of  a sub-Riemannian manifold. One could try to use recent geometric techniques such as  generalized curvature-dimension inequalities or comparison theorems. We are not relying on Cameron-Martin-type results \cite{BaudoinGordinaMelcher2013, BaudoinFengGordina2019} based on these geometric methods in our work, though such results might be useful in the future. Instead, our techniques exploit pathwise properties of the diffusion we consider.

Let $g_{t}$ be a hypoelliptic Brownian motion on the Heisenberg group $\Hei \cong \mathbb{R}^{3}$, and $\WH$ be the path space of $\Hei$-valued continuous curves starting at the identity of $\Hei$.
Note that while the trajectories of $g_{t}$ are continuous curves in $\mathbb{R}^{3}$, their pathwise behaviour reflect the constrained movement of a sub-Riemannian manifold. In order to find the Onsager-Machlup functional for $g_{t}$ with respect to the sup norm, we introduce a notion of  \emph{continuous horizontal paths} to describe the path space of hypoelliptic Brownian motion on $\Hei$. While the notion of horizontal or subunit paths have been a staple in sub-Riemannian geometry and control theory, it usually requires more regularity than just continuity.

The construction of continuous horizontal paths relies on a Wong-Zakai-like approximation of $g_{t}$. This is closely related to one of the central issues in the rough path theory, namely, that the It\^{o} map for L\'evy's area cannot be made continuous as a function of the driving Brownian motion, e.g. \cite[Section 13.6.2]{FrizVictoirBook2010}, \cite[pp. 115-116]{Inahama2019a}, and \cite{BassHamblyLyons2002}. While the space of continuous horizontal paths is not closed under pointwise multiplication, we can define a horizontal version of any continuous path by a map~$\mathcal{K}: D_{\mathcal{K}} \rightarrow \WH$, where  $D_{\mathcal{K}}$ is a subset of $\WH$. This map can be used to defined a semi-metric on $\WH$, namely,
\[
d_{\mathcal{H}} (\gamma , \varphi ) := \max_{0 \leqslant t \leqslant 1} \vert \mathcal{K} (\varphi ^{-1} \gamma ) (t) \vert.
\]

\begin{theorem*}[Theorem \ref{thm.OM}]
Let $\Hei$ be the Heisenberg group, and $g_{t}$ be the hypoelliptic Brownian motion starting at the identity $e \in \Hei$. There exists a finite constant $C\left( \varepsilon \right) >0$ only depending on $\varepsilon$ such that for any $\varphi  \in D_{\mathcal{K}}$

\begin{equation*}
\lim_{\varepsilon \rightarrow 0} \frac{\Prob \left( d_{\mathcal{H}} (g, \varphi ) <\varepsilon \right)}{\Prob \left( d_{\mathcal{H}} (g, e ) <\varepsilon \right)} = \exp\left( -\frac{1}{2} \Vert \pi_H ( \varphi) \Vert^2_{H_{0} \left( \R^2 \right)} \right).
\end{equation*}
On the other hand, if $\varphi \notin D_{\mathcal{K}}$, then for all $\varepsilon>0$ sufficiently small we have that
\[
\Prob \left( d_{\mathcal{H}} (g, \varphi ) < \varepsilon \right)=0.
\]
\end{theorem*}
The domain $D_{\mathcal{K}}$ of this \textit{horizontal} Onsager-Machlup functional can be described in terms of the domain of the Onsager-Machlup functional associated to a two-dimensional Brownian motion as observed in  Remark \ref{rmk.domain}. Moreover, the constant $C\left( \varepsilon \right)$ is given explicitly in terms of the spectral gap of the infinitesimal generator of $g_{t}$ in a metric ball in  $\Hei$ whose existence was proven in \cite{CarfagniniGordina2022b}.

Let us now describe the construction of the map $\mathcal{K}$ in more detail. The trajectories of a diffusion are only continuous, whereas the most general definition of \emph{horizontal curves} assumes absolute continuity. In Section~\ref{s.ContHorizont} we introduce the notion of continuous horizontal paths and prove that the trajectories of the process $g_{t}$ are indeed continuous horizontal paths. Recall that an absolutely continuous curve $\gamma$ is horizontal if the tangent vector at almost every point is horizontal, i.e. $\gamma^{\prime}(t)$ lives in $\mathbb{R}^{2}$ for a.e.~$t$. Let $H_{0} \left( \R^{2} \right)$ be the Cameron-Martin subspace of $\WR$. We first construct a map $T: H_{0} \left( \R^{2} \right) \rightarrow \WH$ with the following property: an absolutely continuous curve $\gamma$ is horizontal if and only if $\gamma \in  T \left( H_{0} \left( \R^{2} \right) \right)$. We then define continuous horizontal curves by extending the map $T$ from $H_{0} \left( \R^{2} \right)$ to $\WR$.

The space of horizontal paths is not closed with respect to the pointwise multiplication in the path space $\WH$. Equation \eqref{e.horizontalization} defines a horizontal version of a continuous path given by a map $\mathcal{K}: D_{\mathcal{K}} \rightarrow \WH$, and whose domain $D_{\mathcal{K}} \subset \WH$ is a subgroup of $\WH$ with respect to the pointwise multiplication. The map $\mathcal{K}$ has the property that $\mathcal{K}(\gamma)$ is horizontal for any $\gamma\in D_{\mathcal{K}}$. We rely on this property to prove Theorem \ref{thm.OM}.

We expect the techniques introduced in this paper to be applicable to a broader class of hypoelliptic diffusions such as Brownian motions in Carnot groups. These stochastic processes can be written explicitly by means of iterated L\'evy's area functionals, making the analysis even more challenging. Adapting our methods to Carnot groups would require the construction of a delicate approximation to a hypoelliptic Brownian motion and we plan to address this problem in the future.

The paper is organized as follows. In Section~\ref{s.OM} we define the Onsager-Machlup functional,  and in Section~\ref{s.Heisenberg} we describe the Heisenberg group $\Hei$ and the corresponding sub-Laplacian and hypoelliptic Brownian motion $g_{t}$. We then study geometric properties of the path space $\WH$ including properties of the maps $\mathcal{K}$ and $T$, and use them to define continuous horizontal curves. In Section~\ref{s.OMfunctional}  we describe the Onsager-Machlup functional for $g_{t}$ with respect to the sup norm, and in Section~\ref{sec5} we prove our main result.


\section{Onsager-Machlup functional on metric spaces}\label{s.OM}

Suppose $\left( E, d_{E} \right)$ is a complete separable metric space. We consider an $E$-valued continuous stochastic process $\{ X_{t} \}_{0\leqslant t \leqslant 1}$ on a probability space $\left( \Omega, \mathcal{F}, \mathbb{P} \right)$ with  $X_{0} =x_{0} \in E$ a.s. We assume that the stochastic process  $\left\{ X_{t} \right\}_{0\leqslant t \leqslant 1}$ is \emph{adapted} to a filtered probability space $\left( \Omega, \mathcal{F}, \mathcal{F}_{t}, \mathbb{P} \right)$, and we assume that the filtered probability space satisfies the \emph{usual conditions}, that is, it is complete and the filtration is right-continuous. Often we choose the smallest such a filtration, namely, a natural filtration.

Denote by $W_{x_{0}}\left( E \right)$ the Wiener space of $E$-valued continuous  functions on $[0,1]$ starting at $x_{0}$. Note that we can view $X_{t}$ as a  $W_{x_{0}}\left( E \right)$-valued random variable, and we denote its law by $\mu$. Given an element $\varphi \in  W_{x_{0}}\left( E \right)$, we are interested in finding the asymptotics of
\begin{align*}
\Prob \left( \max_{0\leqslant t \leqslant 	1} d_{E} \left( X_{t}, \varphi (t) \right) < \varepsilon \right)
\end{align*}
as $\varepsilon$ goes to zero. More precisely, we want to find a real-valued functional $\mathcal{L}$ on $W_{x_{0}} \left( E \right)$ and a constant $C\left( \varepsilon \right)>0$ depending only on $\varepsilon$ such that
\begin{align}\label{e.OM}
\lim_{\varepsilon \rightarrow 0} \frac{1}{C\left( \varepsilon \right)}  \Prob \left( \max_{0\leqslant t \leqslant 	1} d_{E} \left( X_{t}, \varphi (t) \right) < \varepsilon \right)  = \exp \left( \mathcal{L} (\varphi) \right).
\end{align}

\begin{definition}[Onsager-Machlap functional] \label{d.OM}
Suppose there is  a subset $D_{\mathcal{L}}$ of $W_{x_{0}} \left( E \right)$ and a functional  $\mathcal{L}$  such that the limit \eqref{e.OM} exists for any $\varphi \in D_{\mathcal{L}}$. We call
\[
\mathcal{L}: D_{\mathcal{L}} \longrightarrow \R
\]
\emph{the Onsager-Machlup functional} for $X_{t}$ with respect to the $\sup$-norm
\[
\max_{0\leqslant t \leqslant 1} d_{E} \left( \cdot, \cdot \right)
\]
on $W_{x_{0}} \left( E \right)$. The set $D_{\mathcal{L}}$ is called \emph{the domain of the Onsager-Machlup functional} $\mathcal{L}$.
\end{definition}

\begin{example}[Elliptic diffusion] \label{ex.elliptic}
This example is based on \cite{Capitaine2000, TakahashiWatanabe1980}.
Let $X_ t$ be the diffusion process which is a solution to the stochastic differential equation in $\mathbb{R}^{m}$
\begin{equation*}
dX_{t}=\sigma(t, X_{t}) dB_{t} + b(t, X_{t}) dt, \quad X_{0}=x_{0},
\end{equation*}
where $B_{t}$ is an $\R^d$-valued standard Brownian motion, $\sigma=\sigma\left( t, x \right)$ is an $m\times d$ matrix whose entries are smooth functions in $t$ and $x$,  and $b\in \R ^m$ is a smooth function in $t$ and $x$.  In this example we assume that the generator of $X_{t}$ is elliptic, therefore the matrix $\sigma$ is invertible. Then we can view  $\R^m$ as a Riemannian manifold equipped with the metric $g=(\sigma \sigma^{T})^{-1}$, and  $X_{t}$ as a diffusion process on the manifold $M=(\R^m, g)$ with infinitesimal generator $\frac{1}{2}\Delta_M +Z$, where $Z$ is a smooth vector field and $\Delta_M$ is the Laplace-Beltrami operator on $(M, g)$. More precisely, we have
\begin{align*}
& Z_i(x) := b_i(x) + \frac{1}{2} \sum_{j,k=1}^m(\sigma \sigma^{T})_{kj}(x) \Gamma_{kj}^i,
\\
& \Delta_M f = \sum_{i,j=1} ^m(\sigma \sigma^{T})_{ij}\frac{\partial ^2}{\partial x_i\partial x_j} f - \sum_{i,j,k=1}^m (\sigma \sigma^{T})_{ij} \Gamma_{ij}^k \frac{\partial}{\partial x_k}f,
\end{align*}
where $\Gamma_{ij}^k$ are the Christoffel symbols corresponding to the Levi-Civita connection on the Riemannian manifold $(M, g)$.

Let us denote by $d_{M}$ the Riemannian distance on $M$ induced by $g$, and by $TM$ the tangent bundle over $M$. In \cite{FujitaKotani1982, TakahashiWatanabe1980} it is shown that for any smooth curve $\varphi$ in $M$ with $\varphi(0)=x_{0}$ we have
\begin{align*}
& \lim_{\varepsilon \rightarrow 0} \frac{1}{ C\left(  \varepsilon \right) } \Prob \left(  d_{W_{0} (M)} ( X, \varphi) < \varepsilon \right) = \exp \left( \mathcal{L}(\varphi) \right),
\end{align*}
where
\begin{align*}
& d_{W_{0} (M)} ( X, \varphi): =\max_{0\leqslant t \leqslant T } d_M \left(X_{t}, \varphi \left( t \right) \right), \quad C\left( \varepsilon \right)=  \Prob \left( \max_{0 \leqslant t  \leqslant T} \vert B_{t} \vert < \varepsilon \right),
\\
&\mathcal{L}(\varphi) =   \int_{0}^T L \left( \varphi\left( s \right),  \varphi^\prime \left( s \right)  \right)  ds,
\\
& L(p,v)= -\frac{1}{2}||Z_p - v||_p ^2 -\frac{1}{2}\operatorname{div} Z_p + \frac{1}{12} R(p)
\end{align*}
for any $(p,v)\in TM$. Here $\Vert \cdot \Vert_{p}$ denotes the Riemannian norm on $T_{p}M$  and $R(p)$ is the scalar curvature at $p$.

\end{example}

\section{Heisenberg group basics}\label{s.Heisenberg}

\subsection{Heisenberg group as Lie group}\label{subsec3.1}
The Heisenberg group $\Hei$ as a set is  $\R^3\cong \mathbb{R}^{2} \times \mathbb{R}$ with the group multiplication  given by
\begin{align*}
& \left( \mathbf{v}_{1}, z_{1} \right) \cdot \left( \mathbf{v}_{2}, z_{2} \right) := \left( x_{1}+x_{2}, y_{1}+y_{2}, z_{1}+z_{2} + \frac{1}{2}\omega\left( \mathbf{v}_{1}, \mathbf{v}_{2} \right)\right),
\\
& \mathbf{v}_{1}=\left( x_{1}, y_{1} \right), \mathbf{v}_{2}=\left( x_{2}, y_{2} \right) \in \mathbb{R}^{2},
\\
& \text{where }\omega: \mathbb{R}^{2} \times \mathbb{R}^{2} \longrightarrow \mathbb{R}, \; \omega\left( \mathbf{v}_{1}, \mathbf{v}_{2} \right):= x_{1}y_{2}-x_{2} y_{1}
\end{align*}
is the standard symplectic form on $\mathbb{R}^{2}$. The identity in $\Hei$ is $e=(0, 0, 0)$ and the inverse is given by $\left( \mathbf{v}, z \right)^{-1}= (-\mathbf{v},-z)$.

The Lie algebra of $\Hei$ can be identified with the space  $\R^3\cong \mathbb{R}^{2} \times \mathbb{R}$  with the Lie bracket defined by
\[
\left[ \left( \mathbf{a}_{1}, c_{1} \right), \left( \mathbf{a}_{2}, c_{2} \right)  \right] = \left(0,\omega\left( \mathbf{a}_{1}, \mathbf{a}_{2} \right)  \right).
\]
The set $\R^3\cong \mathbb{R}^{2} \times \mathbb{R}$ with this Lie algebra structure will be denoted by $\mathfrak{h} $.

Let us now recall some basic notation for Lie groups. Suppose $G$ is a Lie group, then the left  and right multiplication by an element $k\in G$ are denoted by
\begin{align*}
L_{k}: G \longrightarrow G, &  & g \longmapsto k^{-1}g,
\\
R_{k}: G \longrightarrow G, &  & g \longmapsto gk.
\end{align*}
Recall that  the tangent space $T_{e}G$ can be identified with the Lie algebra $\mathfrak{g}$ of left-invariant vector fields on $G$, that is, vector fields $X$ on $G$  such that $dL_{k} \circ X=X \circ L_{k}$, where $dL_{k}$ is the differential of $L_k$. More precisely, if $A$ is a vector in $T_{e}G$, then we denote by $\widetilde{A}\in \mathfrak{g}$ the (unique) left-invariant vector field such that $\widetilde{A} (e) = A$.  A left-invariant vector field is determined by its value at the identity, namely, $\widetilde{A}\left( k \right)=dL_{k} \circ\widetilde{A}\left( e \right)$. In addition to actions of $G$ on itself by left and right multiplication, we can also consider the action of $G$ on itself by conjugation, namely,  for each $k\in G$ we define the \emph{inner automorphism} by
\begin{align*}
\operatorname{Inn}_{k}\left( g \right):= k^{-1}gk, g \in G.
\end{align*}
Note that  $\operatorname{Inn}_{k}\left( e \right)=e$, and therefore we have an invertible linear map $\left( d\operatorname{Inn}_{k}\right)_{e}: T_{e}G  \longrightarrow T_{e}G$. Using the identification of the Lie algebra $\mathfrak{g}$ with $T_{e}G$ we can introduce the adjoint representation of $G$.

\begin{definition} The adjoint representation of a Lie group $G$ is the representation of $G$ on the group of Lie algebra automorphisms $ \operatorname{Ad}: G \longrightarrow \operatorname{Aut}\left( \mathfrak{g} \right)$  defined by

\[
\operatorname{Ad}_{k} = \left( d\operatorname{Inn}_{k}\right)_{e}.
\]
\end{definition}
For the Heisenberg group $\operatorname{Ad}$ can be described explicitly as follows.

\begin{proposition}\label{p.Differentials}
Let $k= (k_1, k_2, k_3) = (\mathbf{k}, k_3 )$ and $g= (g_1, g_2, g_3) = (\mathbf{g}, g_3 )$ be two elements in $\Hei$. Then, for every $v= \left( v_1, v_2, v_3 \right) = (\mathbf{v}, v_3 )$ in  $T_g\Hei$, the differentials of the left and right multiplication are given by
\begin{align}\label{LeftRightMultDiff}
& dL_{k}: T_g\Hei \longrightarrow T_{k^{-1}g}\Hei,  \notag
\\
& dR_{k}: T_g\Hei \longrightarrow T_{gk}\Hei, \notag
\\
& dL_{k} (v) =  \left( v_1, v_2, v_3 + \frac{1}{2} \omega( \mathbf{v}, \mathbf{k}) \right), \notag \\
& dR_{k} (v) =  \left( v_1, v_2, v_3 + \frac{1}{2} \omega( \mathbf{v}, \mathbf{k}) \right).
\end{align}
If $h= \left( a, b, c \right) = (\mathbf{a}, c ) \in \mathfrak{h}$, then the adjoint representation is given by
\begin{align}\label{differentials}
& \operatorname{Ad}_{k}:  \mathfrak{h} \longrightarrow \mathfrak{h}, \notag
\\
& \operatorname{Ad}_{k} (h) =  \left(a, b, c +\omega( \mathbf{a}, \mathbf{k} )\right).
\end{align}
\end{proposition}

\begin{proof}
Suppose  $\gamma_{g}$ is an integral curve for the vector field $\tilde{h}$ starting at $g$. Then $\gamma_{g}^{\prime} (0)=w$ and therefore
\begin{align*}
& \operatorname{Ad}_{k} \left( h \right) = \left.\frac{d}{dt}\right|_{t=0}(\operatorname{Inn}_{k}  \gamma_{g} \left( t \right))
\\
& = \left( \gamma_1^{\prime}(0),  \gamma_2^{\prime}(0), \gamma_3^{\prime}(0) + \frac{1}{2}  \omega\left( \left(  \gamma_1^{\prime}(0),  \gamma_2^{\prime}(0) \right), \mathbf{k}\right)  \right)
 \\
&= \left( a, b, c + \omega (\mathbf{a}, \mathbf{k} )   \right).
\end{align*}
The rest of the statement can be shown similarly.
\end{proof}

\subsection{Heisenberg group as a sub-Riemannian manifold}
The Heisenberg group $\Hei$ is the simplest non-trivial example of a sub-Riemannian manifold.
We define $X$, $Y$, and $Z$ as the unique left-invariant vector fields satisfying $X_e = \partial_x$, $Y_e = \partial_y$ and $Z_e = \partial_z$, that is,
\begin{align*}
& X = \partial_x - \frac{1}{2}y\partial_z,
\\
& Y = \partial_y + \frac{1}{2}x\partial_z,
\\
&  Z = \partial_z.
\end{align*}
Note that the only non-zero Lie bracket for these left-invariant vector fields is $[X, Y]=Z$, so the vector fields $\left\{ X, Y \right\}$ satisfy H\"{o}rmander's condition. We define the \emph{horizontal distribution} as $\mathcal{H}:= \Span \left\{  X, Y \right\}$ fiberwise, thus making $\mathcal{H}$ a sub-bundle in the tangent bundle $T\Hei$. To finish the description of the Heisenberg group as a sub-Riemannian manifold we need to equip the horizontal distribution  $\mathcal{H}$ with an inner product. For any $p \in \Hei$ we define the inner product $\langle \cdot , \cdot \rangle_{\mathcal{H}_{p}}$ on $\mathcal{H}_{p}$ so that $\left\{ X \left( p \right), Y \left( p \right) \right\}$ is an orthonormal (horizontal) frame at any $p \in \Hei$. Vectors in $\mathcal{H}_{p}$ will be called \emph{horizontal}, and the corresponding norm will be denoted by $\Vert \cdot \Vert_{\mathcal{H}_{p}}$.

In addition, H\"{o}rmander's condition ensures that a natural sub-Laplacian on the Heisenberg group
\begin{equation}\label{e.2.1}
\Delta_{\mathcal{H}} =  X ^2 + Y ^2
\end{equation}
is a hypoelliptic operator by \cite{Hormander1967a}.

We recall now two other important objects in sub-Riemannian geometry, namely, horizontal curves and horizontal gradient. We start by introducing Maurer-Cartan forms on a Lie group $G$.

\begin{notation}\label{n.MaurerCartan}
By $\theta^{l}$ ($\theta^{r}$) we will denote the left (right) Maurer–Cartan form on $G$, i.e. $\theta$ is the
$\mathfrak{g}$-valued $1$-form on $\mathfrak{h}$ defined by
\begin{align*}
& \theta_{k}^{l}\left( v \right):=dL_{k} \left( v \right)
\\
& \theta_{k}^{r}\left( v \right):=dR_{k^{-1}} \left( v \right).
\end{align*}
for any $g \in G, v \in T_{g} G$.
\end{notation}
Note that in general

\begin{equation}\label{MaurerCartanLR}
\theta_{k}^{r}\left( v \right)=\operatorname{Ad}_{k^{-1}}\left(\theta_{k}^{l}\left( v \right) \right).
\end{equation}
For the Heisenberg group Maurer-Cartan forms can be written explicitly by using \eqref{LeftRightMultDiff}  as follows

\begin{align}\label{e.MaurerCartanHeis}
& \theta_{k}^{l}\left( v \right)= \left( v_1, v_2, v_3 + \frac{1}{2} \omega( \mathbf{v}, \mathbf{k}) \right),
\\
& \theta_{k}^{r}\left( v \right)= \left( v_1, v_2, v_3 - \frac{1}{2} \omega( \mathbf{v}, \mathbf{k}) \right). \notag
\end{align}

\begin{definition}\label{d.horizontal}
An absolutely continuous path $t \longmapsto \gamma \left( t \right) \in \Hei,  \; 0 \leq t \leq 1$ is said to be \emph{horizontal} if $\gamma^{\prime}\left( t \right)\in\mathcal{H}_{\gamma\left( t \right)}$ for all $t$, that is, the tangent vector to $\gamma\left(t\right)$ at every point $\gamma \left(t\right)$ is horizontal. Equivalently we can say that $\gamma$ is horizontal if the (left) Maurer-Cartan form $c_{\gamma}\left( t \right):=\theta_{\gamma\left( t \right)}^{l}\left( \gamma^{\prime}\left( t \right)\right) \in \mathcal{H}_{e}$ for a.e. $t$.
\end{definition}
Note that for $\gamma\left( t \right) = \left( \mathbf{x}\left( t \right), z\left( t \right) \right)$ we have
\begin{align}\label{e.MaurerCartan}
& c_{\gamma}\left( t \right)=\theta_{\gamma\left( t \right)}^{l}\left( \gamma^{\prime}\left( t \right)\right)=\left( \mathbf{x}^\prime\left( t\right), z^{\prime}\left( t \right) -\frac{1}{2}\omega( \mathbf{x}\left( t \right), \mathbf{x}^{\prime}\left( t \right) )\right), \notag
\end{align}
where we used Proposition~\ref{p.Differentials}. It is then easy to see that a curve $\gamma$ is horizontal if and only if
\begin{equation}\label{e.horizontal}
z^{\prime}\left( t \right) -\frac{1}{2}\omega( \mathbf{x}\left( t \right), \mathbf{x}^{\prime}\left( t \right) ))=0.
\end{equation}
\begin{definition}\label{d.finiteenergypath}
We say that an absolutely continuous horizontal curve $t \longmapsto \gamma \left( t \right) \in \Hei$, for a.e.  $0\leq t \leq 1$ has \emph{finite energy} if
\begin{equation}\label{e.finite.energy}
\Vert \gamma \Vert_{H\left( \Hei \right)}^{2}
:= \int_{0}^1 \vert c_\gamma \left( s \right) \vert^2_{\mathcal{H}_e}  ds=\int_{0}^1 \vert \theta_{\gamma(s)} \left( \gamma^\prime(s) \right)\vert^2_{\mathcal{H}_{e}} ds < \infty.
\end{equation}
We denote  by $\mathsf{Hor}\left( \Hei\right)$ the \emph{Cameron-Martin space of finite energy horizontal paths} starting at the identity
\begin{align*}
\mathsf{Hor}\left( \Hei\right):= & \left\{ \gamma : [0,1] \longrightarrow \Hei,   \Vert \gamma \Vert_{H\left( \Hei \right)} < \infty,  \gamma\left( 0 \right)=e,
\right.
\\
& \left.
\gamma  \text{ is absolutely continuous and horizontal} \right\}.
\end{align*}
The inner product corresponding to the norm $\Vert \cdot \Vert_{H\left( \Hei \right)}$ is denoted by $\langle \cdot, \cdot \rangle_{H\left( \Hei \right)}$.
\end{definition}

Note that the Heisenberg group as a sub-Riemannian manifold comes with a natural left-invariant distance which we will use to define the Onsager-Machlup functional.

\begin{definition}\label{Dfn.2.2}
For any $g_{1}, g_{2} \in \Hei$ the Carnot-Carath\'{e}odory distance is defined as
\begin{align*}
d_{cc} (g_{1}, g_{2}):= &\inf \left\{  \int_{0}^1  \vert c_\gamma \left( s \right) \vert_{\mathcal{H}_e}  , \right.
\\
& \left. \gamma : [0,1] \longrightarrow \Hei, \gamma(0)=g_{1}, \gamma(1)=g_{2},  \gamma  \text{ is horizontal}  \right\}.
\end{align*}
\end{definition}
Another consequence of  H\"{o}rmander's condition for left-invariant vector fields $X$, $Y$ and $Z$ is that we can apply the Chow–Rashevskii theorem. As a result, given two points in $\Hei$ there exists a horizontal curve connecting them, and therefore the Carnot-Carath\'{e}odory distance is finite on $\Hei$. The Carnot-Carath\'{e}odory distance defined in Definition~\ref{Dfn.2.2} is an example of a control distance related to the left-invariant  vector fields $X$, $Y$ and $Z$. We refer for more details to \cite[Definition 5.2.2]{BonfiglioliLanconelliUguzzoniBook}.

In addition to the Carnot-Carath\'{e}odory distance on the Heisenberg group,  we will use  the following homogeneous distance
\begin{equation}\label{hom.norm}
\rho( g_{1}, g_{2}):= \left(  \Vert \mathbf{x}_{1}-  \mathbf{x}_{2}\Vert^{4}_{\mathbb{R}^{2}}+ \vert z_{1}-z_{2} + \omega(\mathbf{x}_{1}, \mathbf{x}_{2}) \vert^2 \right)^{\frac{1}{4}},
\end{equation}
which is equivalent to the Carnot-Carath\'{e}odory distance, that is, there exist  two positive constants $c$ and $C$ such that
\begin{equation}\label{e.DistEquivalence}
c \rho( g_{1}, g_{2}) \leqslant d_{cc}( g_{1}, g_{2} ) \leqslant C \rho( g_{1}, g_{2})
\end{equation}
for all $g_{1}, g_{2} \in \Hei$. We denote by $\vert \cdot \vert$ the norm on $\Hei$ induced by $\rho$, that is, $\vert g \vert = \rho ( g, e )$ for all $g\in \Hei$. In particular, by the left-invariance of $\rho$ we have that for any $g_{1}, g_{2} \in \Hei$
\begin{equation}\label{e.triangular.ineq}
\vert g_{2}^{-1} g_{1} \vert = \rho\left( g_{2}^{-1} g_{1}, e \right) = \rho \left( g_{1}, g_{2} \right) \leqslant \rho \left( g_{1},e \right) +  \rho \left( g_{2},e \right) = \vert g_{1} \vert +\vert g_{2} \vert.
\end{equation}
This is discussed in a more general setting in \cite[Proposition 5.1.4]{BonfiglioliLanconelliUguzzoniBook}.

\begin{definition}\label{d.HeisenbergBM}
An $\Hei$-valued Markov process $g_{t}$ is called a hypoelliptic Brownian motion if its generator is the sub-Laplacian $\frac{1}{2} \Delta_{\mathcal{H}}$ defined by \eqref{e.2.1}.
\end{definition}
The process $g_{t}$ satisfies a stochastic differential equation as follows. Namely, we can use the Maurer-Cartan form or the differential of left multiplication operator to find a Lie group-valued Brownian motion
\begin{align*}
& \theta_{g_{t}}^{l}\left( dg_{t} \right)=dL_{g_{t}}\left( dg_{t} \right)=\left( dB_{t}, 0 \right),
\\
& g_{0}=e,
\end{align*}
where $B_{t} = \left( B_{1} (t), B_{2} (t) \right)$ is a standard two-dimensional Brownian motion. An explicit solution is given by
\begin{equation}\label{e.HypoBM}
g_{t}:=\left( B_{t}, A_{t} \right),
\end{equation}
where $A_{t}:= \frac{1}{2} \int_{0}^t \omega\left( B_{s}, dB_{s}\right)$ is L\'{e}vy's stochastic area. Note that we used an It\^{o} integral in the definition of $g_{t}$ rather than the Stratonovich integral. However, these two integrals are equal since the symplectic form $\omega$ is skew-symmetric, and therefore  L\'{e}vy's stochastic area functional is the same for both integrals.

\section{Continuous horizontal paths}\label{s.ContHorizont}

The goal of this section is to give a natural notion of continuous horizontal curves which might not be absolutely continuous. First, we introduce the notation that will be used throughout the rest of the paper.

\begin{notation}[Standard Wiener space] We denote by
\[
W_{0}\left( \mathbb{R}^{n} \right)=W_{0}\left( [0, 1], \mathbb{R}^{n} \right)
\]
the space of $\mathbb{R}^{n}$-valued continuous functions starting at $0$. This space is equipped  with the norm

\[
\Vert h \Vert_{W_{0} \left(\R^n\right)}:=\max_{0\leqslant t \leqslant 1} \vert h \left( t \right) \vert_{\R^n}, \quad h\in W_{0} \left( \R^n \right),
\]
and the corresponding distance
\[
d_{W_{0} \left(\R^n\right)} (h,k) =\max_{0 \leqslant t\leqslant 1} \vert h \left( t \right)-k\left( t \right) \vert_{\R^n},
\]
where $\vert \cdot \vert_{\R^n} $ is the Euclidean norm. Moreover, we denote by $H_{0} \left( \R^n \right)$ the Cameron-Martin subspace of  absolutely continuous curves such that
\[
\Vert \gamma \Vert_{H_{0} \left( \R^n \right)}:=\int_{0}^1 \vert \gamma^\prime \left( t \right) \vert^2_{\R^n}ds < \infty.
\]
\end{notation}
\begin{notation}[Wiener space over $\Hei$]\label{d.wiener.space}
We denote by $\WH$ the Wiener space over $\Hei$, i.e. the space of $\Hei$-valued continuous functions starting at the identity in $\mathbb{H}$.
\end{notation}
Once a norm on $\Hei$ is fixed, one can introduce topology on $W_{0} \left( \Hei \right)$ in the following way.
\[
\Vert \eta \Vert_{\WH}:=\max_{0\leqslant t \leqslant 1} \vert \eta \left( t \right) \vert, \quad \eta \in W_{0} \left( \Hei \right),
\]
and the corresponding distance is
\[
d_{\WH}(\eta, \gamma) = \Vert \eta^{-1} \gamma \Vert = \max_{ 0 \leqslant t \leqslant 1 } \vert \eta\left( t \right)^{-1} \gamma \left( t \right) \vert
\]
for any $\eta, \gamma \in W_{0} \left( \Hei \right)$.

\subsection{Absolutely continuous horizontal projections}

We now describe the map that we can informally view as a horizontal projection on the Wiener space $W_{0} \left( \Hei \right)$.

\begin{notation}\label{n.hypo.wiener.measure}
Let $W_{0} \left( \Hei \right)$ be the Wiener space over $\Hei$, and $g_{t}$ be the hypoelliptic Brownian motion defined by Equation \eqref{e.HypoBM}. We denote its law by $\mu$.
\end{notation}

Let us consider the map
\begin{align}
& T: H_{0}\left( \R^2\right) \longrightarrow \WH  \label{dfn.T}
\\
& T(\xi) \left( t \right) := \left(\xi\left( t \right), \frac{1}{2} \int_{0}^t \omega\left( \xi\left( s \right),  \xi^\prime \left( s \right) \right)ds \right), \; \xi \in H_{0}\left( \R^2\right), \notag
\end{align}
and the projections
\begin{align}
& \pi_H : \WH \longrightarrow \WR, \quad \pi_H(\gamma) \left( t \right):= \left( \gamma_1 \left( t \right), \gamma_2 \left( t \right)  \right), \label{dfn.pi}
\\
& \pi_V : \WH \longrightarrow W_{0} \left( \R \right), \quad \pi_V (\gamma) \left( t \right):=  \gamma_3 \left( t \right)  \label{dfn.pi.vertical}
\end{align}
for any $\gamma = \left( \gamma_1, \gamma_2, \gamma_3 \right) \in \WH$. Then the map $ \mathcal{K}:= T\circ \pi_H$ can be written explicitly as
\begin{align}\label{dfn.H}
\mathcal{K}(\gamma ) \left( t \right):= \left(\gamma_1\left( t \right), \gamma_2\left( t \right), \frac{1}{2} \int_{0}^t \omega\left( \gamma (s), \gamma^\prime (s) \right)ds \right),
\end{align}
for any  $\gamma = \left( \gamma_1, \gamma_2, \gamma_3 \right) \in \pi_H^{-1 } \left( H_{0}\left( \R^2\right) \right)$.

\begin{proposition}\label{p.propert.general} Let $\WH$, $\mathsf{Hor} \left( \Hei \right)$ be defined as in Definition~\ref{d.finiteenergypath} and Definition~\ref{n.hypo.wiener.measure}, then

(1)  $\left( \WH, \Vert \cdot \Vert_{\WH} \right)$ is an infinite-dimensional topological group with respect to the pointwise multiplication.

(2) $\mathsf{Hor} \left( \Hei \right)$ is not closed under the group operations in $\left( \WH, \Vert \cdot \Vert_{\WH} \right)$.

(3) $\left( \WH, \WR, \pi_H  \right)$ is a fiber bundle with fibers homeomorphic to $W_{0} \left( \R \right)$.
\end{proposition}

\begin{proof}
(1) Let $\gamma_{1} , \gamma_{2} \in \WH$, then it is clear that $\gamma_{1} \cdot \gamma_{2}^{-1} \in \WH$, so $\WH$ is a group. Therefore, we only need to prove that  the map
\begin{align*}
& G : \left( \WH, \Vert \cdot \Vert_{\WH} \right) \times \left( \WH, \Vert \cdot \Vert_{\WH} \right) \longrightarrow \left( \WH, \Vert \cdot \Vert_{\WH} \right)
\\
& \left( \gamma_1, \gamma_2 \right) \longrightarrow \gamma_1^{-1}  \cdot \gamma_2
\end{align*}
is continuous, that is, for any open neighborhood $U$ of $\gamma_1^{-1}  \cdot \gamma_2$ in the topology of $\WH$, there exists a neighborhood $W$ of $\left( \gamma_1, \gamma_2 \right)$ open in the topology of $\WH \times \WH$ such that $G(W) \subset U$. The uniform topology on $\WH$ is generated by balls, therefore it is enough to prove that for any $r>0$ there exist $r_1, r_2>0$ such that
\begin{equation}\label{eq.continuity}
G \left( B_{r_1} (\gamma_1) \times  B_{r_2} (\gamma_2) \right) \subset B_r  \left( \gamma_1^{-1} \cdot \gamma_2 \right),
\end{equation}
where $B_{r} (\gamma) = \{ \eta \in \WH \; \Vert \eta^{-1} \gamma \Vert_{\WH} < r \}$. For any $r>0$, let $r_{1} = r_{2} := \frac{r}{2}$. Let us prove that for any $h_{i} \in B_{r_i} ( \gamma_{i}) $ for $i=1,2$, then  $G( h_{1}, h_{2}) \in  B_{r}  \left( \gamma_{1}^{-1} \cdot \gamma_{2} \right)$. Indeed, by \eqref{e.triangular.ineq} it follows that
\begin{align*}
& \Vert G(h_{1}, h_{2})^{-1} \gamma_{1}^{-1}\gamma_{2} \Vert_{\WH} = \Vert h_{2}^{-1} h_{1} \gamma_{1}^{-1} \gamma_{2} \Vert_{\WH}
\\
& =  \Vert  h_{1} \gamma_{1}^{-1} \gamma_{2} h_{2}^{-1} \Vert_{\WH} \leqslant \Vert h_{1}^{-1} \gamma_{1} \Vert_{\WH} + \Vert h_{2}^{-1} \gamma_{2} \Vert_{\WH} < \frac{r}{2} +\frac{r}{2} =r,
\end{align*}
which proves \eqref{eq.continuity}.

(2) For $h = \left( \bm{h}, h_3 \right)$ and   $k= \left( \bm{k}, k_3 \right) \in \mathsf{Hor } \left( \Hei \right)$, let $c_h \left( t \right)$ and $c_k \left( t \right)$ be the corresponding Maurer-Cartan forms, that is,
\begin{align*}
& c_h \left( t \right) = \left( \mathbf{h}^\prime\left( t\right), h_3^{\prime}\left( t \right) -\frac{1}{2}\omega( \mathbf{h}\left( t \right), \mathbf{h}^{\prime}\left( t \right) )\right) = \left( \mathbf{h}^\prime \left( t \right), 0 \right),
\\
& c_k \left( t \right) = \left( \mathbf{k}^\prime\left( t\right), k_3^{\prime}\left( t \right) -\frac{1}{2}\omega( \mathbf{k}\left( t \right), \mathbf{k}^{\prime}\left( t \right) )\right) = \left( \mathbf{k}^\prime \left( t \right), 0 \right).
\end{align*}
Then
\begin{align*}
h^{-1}\cdot k \left( t \right) := \left( \mathbf{k}  \left( t \right)-  \mathbf{h} \left( t \right), k_3\left( t \right) - h_3\left( t \right)  + \frac{1}{2} \omega \left( \mathbf{k} \left( t \right), \mathbf{h} \left( t \right) \right) \right),
\end{align*}
and hence
\begin{align*}
& c_{ h^{-1} \cdot k } \left( t \right) = \left( \mathbf{k}^\prime \left( t \right) - \mathbf{h}^\prime  \left( t \right), \right.
\\
& \qquad \qquad\quad  \left. k^\prime_3\left( t \right) - h^\prime_3\left( t \right) + \frac{1}{2} \omega \left( \mathbf{k} \left( t \right), \mathbf{h} \left( t \right) \right)^\prime - \frac{1}{2} \omega \left(\mathbf{k} \left( t \right) - \mathbf{h}  \left( t \right), \mathbf{k}^\prime \left( t \right) - \mathbf{h}^\prime  \left( t \right)  \right) \right)
\\
& \qquad \qquad = \left( \mathbf{k}^\prime \left( t \right) - \mathbf{h}^\prime  \left( t \right) , \omega \left( \mathbf{k} \left( t \right) - \mathbf{h}  \left( t \right), \mathbf{h}^\prime \left( t \right) \right) \right),
\end{align*}
that is, $h^{-1} \cdot k$ is horizontal if and only if $\omega \left( \mathbf{k} \left( t \right) - \mathbf{h}  \left( t \right), \mathbf{h}^\prime \left( t \right) \right) =0$ for all $0\leqslant t \leqslant 1$.

(3) We  need to prove that
\begin{align*}
\pi_H \, : \left(  \WH, \Vert \cdot \Vert_{\WH} \right) \longrightarrow \left(  \WR, \Vert \cdot \Vert_{\WR} \right)
\end{align*}
is a continuous surjective map, and that for any $\psi \in \WR$, the fiber $\left(   \pi_H^{-1} (\psi ) , \Vert \cdot \Vert_{\WH} \right)$ is homeomorphic to $\left( W_{0} \left( \R \right) , \Vert \cdot \Vert_{W_{0} \left( \R \right)} \right)$. First, by Definition \eqref{dfn.pi} it is clear that $\pi_{H}$ is surjective, and since $\Vert \pi_H (\varphi) \Vert_{\WR}\leqslant \Vert \varphi \Vert_{\WH}$ for any $\varphi \in \WH$, we have that $\pi_H$ is continuous in the uniform topology.

Let us now prove that for any  $\psi \in \WR$, the spaces $\pi_H^{-1} (\psi )$ and $W_{0} \left( \R \right)$ are homeomorphic when both endowed with the uniform topology. First, note that
\begin{align*}
& \pi_H^{-1} (\psi ) := \{ \varphi \in \WH : \, \pi_H (\varphi) = \psi  \}
\\
& = \{  \left(\psi, \varphi_3  \right) \in \WH \; \text{ for some } \;  \varphi_3 \in W_{0}\left( \R \right)  \}.
\end{align*}
For any $\gamma \in  \pi_H^{-1} (\psi )$, let us consider the vertical projection restricted to $\pi_H^{-1} (\psi)$, that is,
\begin{align} \label{eqn.homeo}
& \pi_V :  \pi_H^{-1} (\psi) \longrightarrow  W_{0} \left( \R \right),
\\
& \qquad  \gamma \longmapsto \pi_V(\gamma). \notag
\end{align}
We claim that $\pi_V$ is a homeomorphism between  $\left(  \pi_H^{-1} (\psi ) , \Vert \cdot \Vert_{\WH} \right)$  and $\left( W_{0} \left( \R \right) , \Vert \cdot \Vert_{W_{0} \left( \R \right)} \right)$. Note that $\pi_V$ is a bijection. Indeed, if $\gamma := \left( \psi, \gamma_3 \right)$ and  $\eta := \left( \psi, \eta_3 \right)  \in \pi_H^{-1} (\psi)$ satisfy $\pi_V(\gamma) = \pi_V(\eta) $, then $\gamma_3 = \eta_3$  and hence $\gamma = \eta$, proving that $\pi_V$ restricted to $\pi_H^{-1} (\psi)$ is injective. On the other hand, for any  $f \in W_{0} \left( \R \right)$ set $\psi_f := \left( \psi, f\right) \in \WH$. Then $\pi_V( \psi_f) =f$, proving that  $\pi_V$ is surjective.  Continuity of $\pi_V$ can be proved similarly to the continuity of $\pi_H$. The proof is then completed once we show that  the map
\begin{align}
\pi_V^{-1}:& W_{0} \left( \R \right)  \longrightarrow  \pi_H^{-1} (\psi)
\\
&  f \longmapsto \psi_{f} \notag
\end{align}
is continuous, that is, we need to show that for any $r>0$ there exists an $r_1>0$ such that
\begin{align}\label{eqn.contin.inverse}
& \pi_V^{-1} \left( B_{r_1}^{W_{0}\left( \R \right)} (f)  \right) \subset B_r^{\WH} \left( \pi_V^{-1}(f) \right)  \cap \pi^{-1}_H (\psi).
\end{align}
For any $r>0$, set $r_{1}:= r^{2}$. Suppose $g \in  B_{r_1}^{W_{0}\left( \R \right)} (f)$, then
\[
 \left( \pi_{V}^{-1}  (g) \left( t \right) \right)^{-1}  \pi_{V}^{-1}(f) \left( t \right)  = \left( 0,0, f\left( t \right) - g\left( t \right) \right),
\]
and hence it follows that
\begin{align*}
& \Vert\left( \pi_V^{-1}  (g)  \right)^{-1}  \pi_V^{-1}(f)  \Vert_{\WH} := \max_{0 \leqslant t \leqslant 1} \vert \left( \pi_V^{-1}  (g) \left( t \right) \right)^{-1}  \pi_V^{-1}(f) \left( t \right)     \vert
\\
& = \max_{0 \leqslant t \leqslant 1} \vert f\left( t \right) - g\left( t \right) \vert^{\frac{1}{2}} = \Vert f-g  \Vert_{W_{0} \left( \R \right) }^{\frac{1}{2}} < r_1 ^{\frac{1}{2}} = r,
\end{align*}
and \eqref{eqn.contin.inverse} is proven.

\end{proof}

\begin{proposition}[Properties of $T$]\label{p.propert.T}
Let $T$ be given by \eqref{dfn.T}, then

(1) $T$ is not surjective, that is,  $T\left(H_{0} (\R^{2}) \right) \varsubsetneqq \WH$.

(2) The map
\begin{align*}
& T : \left(  H_{0}\left( \R^2 \right) , \Vert \cdot \Vert_{ H_{0}\left( \R^2 \right)}  \right) \longrightarrow \left( \mathsf{Hor} \left( \Hei \right) , \Vert \cdot \Vert_{\mathsf{Hor} \left( \Hei \right)} \right)
\end{align*}
is an isometry with the inverse given by
\[
\left.\pi_{H}\right|_{\mathsf{Hor} \left( \Hei \right) }:  \left( \mathsf{Hor} \left( \Hei \right) , \Vert \cdot \Vert_{\mathsf{Hor} \left( \Hei \right)} \right) \longrightarrow    \left(  H_{0}\left( \R^2  \right) , \Vert \cdot \Vert_{ H_{0}\left( \R^2 \right)}  \right).
\]
\end{proposition}

\begin{proof}
(1)  Recall that $T\left( H_{0}\left( \R^2\right) \right)= \mathsf{Hor} \left( \Hei \right)$  by \eqref{dfn.T} and by Definition~\ref{d.finiteenergypath}. Observe that for $\gamma(t):= (t, t^{2}, 0)$ we have $\gamma \in \WH$ but not in $\mathsf{Hor} \left( \Hei \right)$.

(2)  It is clear that $T$ and $\left.\pi\right|_{\mathsf{Hor} \left( \Hei \right)}$ are inverses of each other. Indeed, if $\gamma \in  \mathsf{Hor} \left( \Hei \right)$, then by \eqref{dfn.T} and \eqref{dfn.pi} we have that $\pi_{H} (\gamma) \in H_{0} \left( \R^2 \right)$ and $T(\pi_{H} (\gamma)) = \gamma$, proving that $T$ is onto  $\mathsf{Hor} \left(  \Hei \right)$. Let $\psi_{1}, \psi_{2}\in  H_{0} \left( \R^2 \right)$ such that $T(\psi_{1}) = T(\psi_{2})$, then by \eqref{dfn.T} we have that $\psi_{1}= \psi_{2}$ proving that $T$ is injective. Moreover, by Definition~\ref{d.finiteenergypath} we see that
\begin{align*}
& \Vert  T( \psi  ) \Vert_{\mathsf{Hor} \left(  \Hei \right)}  = \Vert  \psi \Vert_{H_{0} \left( \R^2 \right)} \; \text{for any} \;   \psi \in H_{0} \left( \R^2 \right),
\\
&  \Vert  \pi( \varphi ) \Vert_{H_{0} \left( \R^2 \right)   }  = \Vert \varphi \Vert_{\mathsf{Hor} \left(  \Hei \right)} \; \text{for any} \;   \varphi \in \mathsf{Hor} \left(  \Hei \right),
\end{align*}
which concludes the proof.
\end{proof}

Below we list properties of the map $\mathcal{K}$ defined by \eqref{dfn.H}, the proof easily follows from the definition of $\pi_H, T$, and $\mathcal{K}$.

\begin{proposition}[Properties of $\mathcal{K}$]\label{p.propert.H} Let $\mathcal{K}$ be the map given by \eqref{dfn.H}. Then

(1)  $\mathcal{K}(\gamma)$ is defined for any $\gamma \in D_{\mathcal{K}}:=  \pi_H^{-1} \left( H_{0}\left( \R^2\right) \right)$  and $\mathcal{K}: D_{\mathcal{K}} \longrightarrow \WH$ is not surjective, that is, $\mathcal{K} \left( D_{\mathcal{K}}   \right)  = \mathsf{Hor} \left( \Hei \right) \varsubsetneqq \WH$, and  $D_{\mathcal{K}}$ is a subgroup of $\WH$.

(2) We have that $\left.\mathcal{K}\right|_{\mathsf{Hor} \left( \Hei \right)} = id_{\mathsf{Hor} \left( \Hei \right)}$ and $\pi_H \circ T= id_{H_{0}\left( \R^2\right) }$. Hence,
\begin{align*}
& \left.\mathcal{K}\right|_{\mathsf{Hor} \left( \Hei \right)}: \left( \mathsf{Hor} \left( \Hei \right)  , \Vert \cdot \Vert_{ \mathsf{Hor} \left( \Hei \right)}  \right) \longrightarrow \left( \mathsf{Hor} \left( \Hei \right), \Vert \cdot \Vert_{\mathsf{Hor} \left( \Hei \right)} \right)
\end{align*}
is an isometry, and therefore continuous.

(3)  If  $\gamma_1, \gamma_2 \in D_{\mathcal{K}}$, then
\begin{align*}
& \mathcal{K} \left( \mathcal{K} (\gamma_1)  \cdot \mathcal{K}(\gamma_2) \right) = \mathcal{K}\left( \gamma_1 \cdot \gamma_2 \right).
\end{align*}
\end{proposition}

\subsection{Admissible approximations}

Let $B_{t}$ be an $\R^{2}$-valued standard Brownian motion and
\[
A_{t}:= \frac{1}{2} \int_{0}^{t} \omega\left( B_{s}, dB_{s} \right)
\]
the corresponding L\'evy's area. Our goal is to extend $T$ to a measure-preserving isomorphism defined on the space of sample paths of $B_{t}$.  Let $\{ B_{\delta} \}_{\delta >0}$ be an approximation of $B_{t}$, that is,
\begin{equation}\label{eqn.convergence.Bm}
\left[ \max_{0\leqslant t \leqslant 1} \vert B_{\delta} (t) - B_{t} \vert^{2}_{\R^{2}} \right] \rightarrow 0 \;
\text{as} \; \delta \rightarrow 0,
\end{equation}
such that
\begin{equation}\label{eqn.convergence.Levyarea}
\left[ \max_{0\leqslant t \leqslant 1} \vert A_{\delta} (t) - A_{t} \vert^{2}_{\R} \right] \rightarrow 0 \;
\text{as} \; \delta \rightarrow 0,
\end{equation}
where
\begin{equation}\label{approximation.LevyArea}
A_\delta\left( t \right) := \frac{1}{2} \int_{0}^t \omega \left( B_{\delta} (s), B^{\prime}_{\delta}(s) \right)ds.
\end{equation}
An example of such approximation is given by
\begin{align*}
B_{i,\delta} \left( t \right):= B_i (k \delta ) + f_i \left( \frac{t-k\delta}{\delta} \right)  \left( B_i (k\delta +\delta ) - B_i(k\delta)   \right),   k \delta \leqslant t <  (k+1)\delta ,
\end{align*}
where $f_i$ $i=1, 2$ are  differentiable functions on $[0,1]$ such that $f_i(0)=0$ and $f_i(1)=1$,  \cite[Theorem 7.1]{IkedaWatanabe1989}.

\begin{definition} We call an approximation to $B_{t}$ satisfying \eqref{eqn.convergence.Levyarea} an \emph{admissible approximation}.
\end{definition}

\begin{remark}\label{r.levyarea.approx}
As pointed out in \cite[Theorem 7.1,  p. 486]{IkedaWatanabe1989} and \cite{Sugita1992}, there is no canonical way to approximate L\'{e}vy's area $A_{t}$. Indeed, one can construct a smooth approximation $\{ \tilde{B}_\delta\}_{\delta >0}$ to $B_{t}$ that does not satisfy \eqref{eqn.convergence.Levyarea}.  This is closely related to the fact that the map $B_{t} \longmapsto A_{t}$ is not continuous with respect to any Banach norm preserving the Gaussian structure of the Wiener space as pointed out in \cite{Sugita1989}. It is not surprising that this is one of the central issues in the theory of rough paths as pointed out in \cite[Section 13.6.2]{FrizVictoirBook2010} and \cite{Inahama2019a}.
\end{remark}

We now use admissible approximations to extend the horizontal projections to elements in the Wiener space $\WH$. Note that for each $\delta$ and $\omega$  the curve $t \rightarrow B_\delta \left( t \right)$ is in $H_{0} \left( \R^2\right)$, and hence  $t \rightarrow T( B_\delta) \left( t \right) = \left( B_\delta \left( t \right), A_\delta\left( t \right) \right)$ is a well-defined element in $\WH$. By \cite[Theorem 3.3]{Carfagnini2021}
\begin{align*}
\lim_{\delta\rightarrow 0} \E \left[ d_{\WH} \left( T(B_\delta), g \right)^2 \right] =0,
\end{align*}
that is, $T(B_\delta)$ converges to $g$ in $L^2\left(\Omega, \Prob \right)$, and hence it converges in probability as well. Therefore, we can extend $T$ to a map $\widehat{T} : \WR \longrightarrow \WH$ by setting $\widehat{T}(B)_{t}=g_{t}$.

\subsection{Continuous horizontal paths}
\begin{notation}\label{n.horizontalcontinuous}
We denote by $W^\mathcal{H}_{0} := \widehat{T}\left( \WR \right)$ the set of continuous horizontal curves.
\end{notation}

Note that, similarly to $\mathsf{Hor} \left( \Hei \right)$, $W^\mathcal{H}_{0} $ is not a subgroup of $\WH$.

\begin{remark}\label{r.measupreserviso}
Let $\nu := \text{law} \, (B)$ and $\mu := \text{law} \,  (g)$, then
\[
\widehat{T}: \left( \WR , \nu \right)  \longrightarrow \left( W^\mathcal{H}_{0}, \mu \right)
\]
is a measure-preserving isomorphism. Indeed, for any Borel subset $A$ of $W_{0}^{\mathcal{H}}$ we have that
\begin{align*}
\mu (A) = \Prob \left( g \in A \right) = \Prob \left( \widehat{T}(B) \in A \right) = \Prob\left( B \in \widehat{T}^{-1} (A) \right) =\nu \left( \widehat{T}^{-1} (A) \right).
\end{align*}
\end{remark}

\begin{definition}\label{d.horizontalization}
Let $\mathcal{K}$ be given by Equation \eqref{dfn.H}, and $\gamma \in D_{\mathcal{K}}$.  Then we refer to
\begin{equation}\label{e.horizontalization}
\mathcal{K} (\gamma)= T\left( \pi_H (\gamma) \right) = \left( \gamma_1, \gamma_2, \frac{1}{2} \int_{0}^\cdot \omega\left( \gamma(s), \gamma^\prime(s) \right)ds  \right)  \in W_{0}^{\mathcal{H}}
\end{equation}
as the \emph{horizontal version} of $\gamma$.
\end{definition}

We abuse notation and denote by $\mathcal{K}$ the map $\widehat{T} \circ \pi_H : \WH \longrightarrow W_{0}^\mathcal{H}$. Note that $\mathcal{K} (g) = g$. By Proposition~\ref{p.propert.general} we know that $\mathsf{Hor}\left( \Hei \right)$ is not closed under the group multiplication, and neither is $W_{0}^\mathcal{H}$. In particular,  $\varphi^{-1}g \notin W_{0}^\mathcal{H}$ even if $\varphi \in \mathsf{Hor} \left( \Hei \right)$ and $g_{t}= T(B)_{t} \in W_{0}^\mathcal{H}$, but we can consider the stochastic  process
\begin{align}\label{e.horizontalprocess}
& \mathcal{K} \left(\varphi^{-1}g \right)_{t}
\\
& =\left( B_{t}- \pi_{H}(\varphi) \left( t \right), \frac{1}{2} \int_{0}^t \omega \left( B_s -\pi_{H}(\varphi) (s), dB_s - \pi_{H}(\varphi^\prime) (s)ds \right)  \right) \in W_{0}^\mathcal{H},
\notag
\end{align}
which we view as a horizontal version  of ~$\varphi^{-1}\left( t \right)g_{t}$. Using the map $\mathcal{K}$ we can define a semi-metric on $\WH$, namely,
\[
d_{\mathcal{H}} (\gamma, \varphi ) := \max_{0 \leqslant t \leqslant 1} \vert \mathcal{K} (\varphi ^{-1} \gamma ) (t) \vert.
\]
Note that the semi-metric $d_{\mathcal{H}}$ is left-invariant, that is,
\[
d_{\mathcal{H}} ( \varphi^{-1} \gamma, e )= d_{\mathcal{H}} (\gamma, \varphi),
\]
and it is consistent with the underlined geometric structure of $\Hei$.

\section{Onsager-Machlup for a hypoelliptic diffusion on the Heisenberg group}\label{s.OMfunctional}
Now we are ready to find the Onsager-Machlup functional for the horizontal version of $\varphi^{-1}\left( t \right)g_{t}$ which gives us the asymptotics of
\[
\Prob \left( d_{\mathcal{H}} (g , \varphi) <  \varepsilon \right),
\]
as $\varepsilon \rightarrow 0$. Recall that by Proposition~\ref{p.propert.H} the map $\mathcal{K} :D_{\mathcal{K}} \longrightarrow \WH$ is defined on the set $D_\mathcal{K} = \pi_H^{-1} \left( H_{0} \left( \R^2 \right) \right)$.

\begin{theorem}\label{thm.OM}
Let $\Hei$ be the Heisenberg group, and $g_{t}$ be the horizontal Brownian motion starting at the identity. There exists a finite constant $C\left( \varepsilon \right) >0$ only depending on $\varepsilon$ such that, for any $\varphi  \in D_{\mathcal{K}}$,
\begin{equation}\label{e.idk}
\lim_{\varepsilon \rightarrow 0} \frac{1}{ C\left( \varepsilon \right)} \Prob \left(d_{\mathcal{H}} (g , \varphi)< \varepsilon \right) = \exp\left( -\frac{1}{2} \Vert \pi_H ( \varphi) \Vert^2_{H_{0} \left( \R^2 \right)} \right).
\end{equation}
On the other hand, if $\varphi \notin D_{\mathcal{K}}$ then for all $\varepsilon>0$ sufficiently small we have that
\[
\Prob \left( d_{\mathcal{H}} (g , \varphi) < \varepsilon \right)=0.
\]
\end{theorem}
This theorem means that the \emph{Onsager-Machlup functional} is given by
\[
\mathcal{L} (\varphi):=-\frac{1}{2} \Vert \pi_{H} (\varphi) \Vert^2_{H_{0} \left( \R^2 \right)}, \varphi \in D_\mathcal{K}.
\]
If $\varphi \notin D_{\mathcal{K}}$ , then $\mathcal{L} (\varphi) =-\infty$. In particular, we have that $D_{\mathcal{L}} = D_{\mathcal{K}}$.

\begin{remark}\label{rmk.domain}
Let $B_{t}$ be a standard Brownian motion in $\R^{n}$. Then by the Cameron-Martin-Girsanov theorem it is easy to see that the domain of its Onsager-Machlup functional with respect to the sup-norm is given by
\begin{align*}
D_{\mathcal{L}} = H_{0} \left( \R^{n} \right),
\end{align*}
the Cameron-Martin space of finite energy paths. Theorem~\ref{thm.OM} means that the domain of the Onsager-Machlup functional with respect to the semi-metric $d_{\mathcal{H}}$ is the pre-image under $\pi_{H}$ of the domain of the Onsager-Machlup functional of $B_{t}$.
\end{remark}

\begin{remark}\label{rmk.constant}
The constant $C\left( \varepsilon \right)$ in \eqref{e.idk} is explicit. More precisely, in the proof of Theorem ~\ref{thm.OM} we will show that
\[
C\left( \varepsilon \right) = \Prob \left( d_{\mathcal{H}} (g , e) < \varepsilon \right) = \Prob \left( \max_{0\leqslant t \leqslant 	1} \vert g_{t} \vert < \varepsilon \right),
\]
and hence by \cite[Theorem 5.3]{CarfagniniGordina2022b} we have that $C\left( \varepsilon \right) \approx \exp \left( - \frac{\lambda}{\varepsilon^2} \right)$, where $\lambda$ is the spectral gap of $-\frac{1}{2} \Delta_{\mathcal{H}}$ on the unit ball $\{x \in \Hei : \vert x \vert <1 \}$. An explicit upper and lower bound on $\lambda$ can be found in \cite[Theorem 5.6]{CarfagniniGordina2022}. Therefore Theorem~\ref{thm.OM} can be interpreted as giving the following approximation

\begin{align*}
\Prob \left(d_{\mathcal{H}} (g, \varphi ) < \varepsilon \right)  \approx \exp \left( - \frac{\lambda}{\varepsilon^2} \right)  \exp\left(-\frac{1}{2} \Vert \pi_{H} (\varphi) \Vert^2_{H_{0} \left( \R^2 \right)} \right),
\end{align*}
which is consistent with the elliptic setting in Example~\ref{ex.elliptic}.
\end{remark}

\section{Proof of Theorem~\ref{thm.OM}}\label{sec5}
First, we describe  the Maurer-Cartan forms induced by the processes $\varphi^{-1}\left( t \right) g_{t}$ and its horizontal version $\mathcal{K}  \left(\varphi^{-1}g \right)_{t}$.
\begin{proposition}
Let $\varphi\in D_{\mathcal{K}}$ with $\varphi(0)=e\in \Hei$. Then the process $\mathcal{K}  \left(\varphi^{-1}g \right)_{t}$ defined by \eqref{e.horizontalprocess} satisfies the equation
\begin{align}\label{e.SDEg.phi}
& \int_{0}^{t} \theta^l _{x_s}\left( dx_{s} \right) = \left(B_{t} - \pi_{H}(\varphi) \left( t \right), 0 \right),
\\
& x_{0} = e. \notag
\end{align}
Moreover, if $\varphi = \left( \pi_{H}(\varphi), \varphi_{3} \right) \in D_{\mathcal{K}}$ and  $\varphi_3$ is absolutely continuous, then the process $\varphi^{-1}\left( t \right) g_{t}$ satisfies
\begin{align}\label{algebraicSDE}
& \int_{0}^{t} \theta^l _{y_s}\left( dy_{s} \right) =
\left( B_{t} -\pi_{H}(\varphi)\left( t\right),  \frac{1}{2} \int_{0}^{t} \omega \left( B_{s},\pi_{H}(\varphi)^{\prime} (s) \right) ds
\right.
\\&
\left. - \frac{1}{2} \int_{0}^{t} \omega \left( \pi_{H}(\varphi)(s),\pi_{H}(\varphi)^{\prime} (s) \right) ds  -\varphi_{3}\left( t \right)  \right),
\notag
\\
& y_{0} = e.  \notag
\end{align}

\end{proposition}

\begin{proof}
Let us prove \eqref{algebraicSDE}.  We use the definition of the Maurer-Cartan form \eqref{e.MaurerCartanHeis}, and \eqref{MaurerCartanLR}, and the fact that $\varphi$ is absolutely continuous to see that almost everywhere
\begin{align*}
& \theta^l _{y_{t}}\left( dy_{t} \right)=\theta^l _{g_{t}}\left( dg_{t} \right)  - \operatorname{Ad}_{g_{t} }\theta^r _{\varphi\left( t \right)}\left( \varphi^{\prime}\left( t \right)\right)dt
\\
&  = (dB_{t},0) -  \operatorname{Ad}_{g_{t} }\theta^r _{\varphi\left( t \right)}\left( \varphi^{\prime}\left( t \right)\right)dt
\\
& =(dB_{t},0) -  \operatorname{Ad}_{g_{t} }\left( \pi_{H}(\varphi)^\prime\left( t\right), \varphi_{3}^{\prime}\left( t \right) +\frac{1}{2}\omega(  \pi_{H}(\varphi) \left( t \right), \pi_{H}(\varphi)^{\prime}\left( t \right) )\right)dt,
\end{align*}
where $y_{t} = \varphi(t)^{-1} g_{t}$. Now we can use  \eqref{differentials} and the explicit form of the Brownian motion $g_{t}$ \eqref{e.HypoBM} to conclude that

\begin{align*}
& \theta^l _{y_{t}}\left( dy_{t} \right)
\\
& =(dB_{t},0)
\\
& -  \left(  \pi_{H}(\varphi)^\prime\left( t\right), \varphi_{3}^{\prime}\left( t \right) +\frac{1}{2}\omega( \pi_{H}(\varphi)\left( t \right), \pi_{H}(\varphi)^{\prime}\left( t \right) )+\omega(\pi_{H}(\varphi)^{\prime}\left( t \right),  B_{t} )\right)dt
\\
& =(dB_{t},0) -  \left(\pi_{H}(\varphi)^\prime \left( t\right), \varphi_{3}^{\prime}\left( t \right) +\omega\left(\pi_{H}(\varphi)^{\prime} \left( t \right),  B_{t}-\frac{1}{2} \pi_{H}(\varphi)\left( t \right) \right) \right)dt,
\end{align*}
and \eqref{algebraicSDE} follows. Similarly, if $x_{t}:=  \mathcal{K}  \left(\varphi^{-1}g \right)_{t}$ we can show that
\begin{align*}
\theta^l _{x_{t}}\left( dx_{t} \right)= \left( dB_{t} - \pi_{H}(\varphi)^\prime \left( t \right) dt, 0\right),
\end{align*}
which completes the proof.
\end{proof}

\begin{remark}\label{r.sde.girsanov}
By \eqref{algebraicSDE} we see that the Maurer Cartan form of the process $\varphi^{-1}(t)g_{t}$ is given by
\[
 \left( dB_{t} -\pi_{H}(\varphi)^\prime \left( t\right) dt ,- \varphi_{3}^{\prime}\left( t \right)dt - \omega\left( \pi_{H}(\varphi)^{\prime} \left( t \right),  B_{t}-\frac{1}{2} \pi_{H}(\varphi)\left( t \right) \right) dt \right),
\]
and it contains a non-linear drift. On the other hand, the Maurer-Cartan form of $\mathcal{K} \left(\varphi^{-1}g \right)_{t}$ is given by
\[
\left( dB_{t} - \pi_{H}(\varphi)^\prime\left( t \right) dt, 0\right),
\]
which only has a linear drift and allows us to use the Cameron-Martin-Girsanov Theorem as shown in the proof of Theorem~\ref{thm.OM}.
\end{remark}

Before proceeding to the proof of Theorem~\ref{thm.OM}, we need the following lemma whose proof can be found in \cite[ pp. 536-537]{IkedaWatanabe1989}.

\begin{lemma}[pp. 536-537 in \cite{IkedaWatanabe1989}]\label{Lemma3.1}
Let $I_1, \ldots, I_n$ be $n$ random variables on a probability space $\left( \Omega, \mathcal{F}, \Prob \right)$. Let $\left\{  A_\varepsilon \right\}_{0<\varepsilon <1}$ be a family of events in $\mathcal{F}$ and $a_1, \ldots , a_n$ be $n$ numbers. If for every real number $c$ and every $1\leqslant i \leqslant n$
\[
\limsup_{\varepsilon \rightarrow 0} \E \left[ \exp (c\, I_i)\, |  A_\varepsilon   \right] \leqslant \exp (c\,a_i),
\]
then
\[
\lim_{\varepsilon\rightarrow 0} \E \left[ \exp\left(\sum_{i=1}^n I_i\right)  |  A_\varepsilon  \right] = \exp \left( \sum_{i=1}^n a_i \right).
\]
\end{lemma}

\begin{proof}[ Proof of Theorem~\ref{thm.OM}]
Let $\varphi = (\pi_{H}(\varphi), \varphi_{3} ) \in D_{\mathcal{K}}$, that is,   $\pi_H (\varphi) \in H_{0}\left( \R^2 \right) $. Therefore by the Cameron-Martin-Girsanov Theorem there exists a probability measure $\Q^\varphi$ such that the process $B^\varphi_{t}:=B_{t} +\pi_{H}(\varphi)\left( t \right)$ is a Brownian motion under $\Q^\varphi$. More precisely there exists an exponential martingale $\mathcal{E}^\varphi$ such that
\[
\Q^\varphi (A) = \E \left[ \mathcal{E}^\varphi \mathbbm{1}_A  \right], \; \text{for any} \; A \in \mathcal{F},
\]
where $\mathcal{E}^\varphi= \exp \left( -\int_{0}^1 \langle \pi_{H}(\varphi)^\prime (s), dB_s \rangle_{\R^2} ds  - \frac{1}{2} \int_{0}^1 \vert \pi_{H}(\varphi)^\prime (s) \vert^2_{\R^2} ds \right)$.  Note that
\begin{align*}
& d \left( B_{t} - \pi_{H}(\varphi) \left( t \right) \right) = dB_{t} -  \pi_{H}(\varphi)^\prime \left( t \right) dt, \; \text{and}
\\
& dB_{t} = dB^{\varphi}_{t} - \pi_{H}(\varphi)^\prime \left( t \right) dt,
\end{align*}
that is, the law of $B_{t} - \pi_{H}(\varphi) \left( t \right) $ under $\Prob$ is the same as the law of $B_{t}$ under $\Q^\varphi$. Thus,
\begin{align*}
&\mathcal{K} \left( \varphi^{-1}g \right)_{t}  = \left( B_{t}-\pi_{H}(\varphi)\left( t \right), \frac{1}{2} \int_{0}^t \omega \left( B_s -  \pi_{H}(\varphi) (s), dB_s - \pi_{H}(\varphi)^\prime (s)ds \right)  \right)
\\
& = \left( B^\varphi_{t}, \frac{1}{2}\int_{0}^t \omega \left( B^\varphi _s, dB^\varphi_s \right) \right),
\end{align*}
and hence the law of $\mathcal{K}  \left(\varphi^{-1}g \right)_{t}$  under $\Prob$ is the same as the law of $g_{t}$ under $\Q^\varphi$. It the follows that
\begin{align*}
&\Prob \left( d_{\mathcal{H}} (g , \varphi )  < \varepsilon \right)  = \Prob \left( \max_{0\leqslant t \leqslant 1} \vert  \mathcal{K} \left(\varphi^{-1}g \right)_{t}  \vert < \varepsilon \right)  = \Q^\varphi \left( \max_{0\leqslant t \leqslant 1} \vert g_{t} \vert < \varepsilon \right)
\\
& = \E \left[ \mathcal{E}^\varphi_1  \mathbbm{1}_{\left\{ \Vert g \Vert_{\WH} < \varepsilon \right\}}   \right] = \Prob \left( \max_{0\leqslant t \leqslant 1} \vert g_{t} \vert < \varepsilon \right) \E \left[ \mathcal{E}^\varphi_1 \, \left|  \, \max_{0\leqslant t \leqslant 1} \vert g_{t} \vert < \varepsilon  \right.  \right],
\end{align*}
that is,
\begin{align}\label{e.what.to.prove}
& \frac{ \Prob \left( d_{\mathcal{H}} (g , \varphi ) < \varepsilon \right) }{ \Prob \left( d_{\mathcal{H}} (g , e )  < \varepsilon \right) } =  \E \left[ \mathcal{E}^\varphi_1 \,  \left|  \,  \max_{0\leqslant t \leqslant 1} \vert g_{t} \vert < \varepsilon  \right.  \right].
\end{align}
It is then clear that we can take $ C\left( \varepsilon \right) := \Prob \left(  d_{\mathcal{H}} (g , e )  < \varepsilon \right)  $ in Theorem~\ref{thm.OM}. Let us now prove that
\begin{equation*}
\lim_{\varepsilon \rightarrow 0}  \E \left[ \mathcal{E}^\varphi_1 \,  \left|  \,  \max_{0\leqslant t \leqslant 1} \vert g_{t} \vert < \varepsilon  \right.   \right] = \exp \left( -\frac{1}{2} \Vert \pi_H( \varphi ) \Vert_{H_{0} \left( \R^2 \right)}^2 \right).
\end{equation*}
Since $\mathcal{E}^\varphi= \exp \left( -\int_{0}^1 \langle \pi_{H}(\varphi)^\prime (s), dB_s \rangle_{\R^2} ds  - \frac{1}{2}  \Vert \pi_H ( \varphi )\Vert^2_{ H_{0} \left( \R^2 \right)} \right)$, by Lemma~\ref{Lemma3.1}, it is enough to show that for any real number $c$ and $i=1,2$
\[
\limsup_{\varepsilon \rightarrow 0} \E \left[ \exp\left( - c\int_{0}^1  \pi_{H}(\varphi)^\prime_i (s) dB_i(s) \right) \, \left| \, \max_{0
\leqslant t \leqslant 1} \vert g_{t} \vert < \varepsilon \right]  \leqslant 1.  \right.
\]
This follows by \cite[p. 654]{SheppZeitouni1992} with $C:=\{ \mathbf{x} \, | \, \Vert g \Vert < \varepsilon \}$ since $\pi_{H}(\varphi)\in H_{0} \left( \R^2 \right)$.

If $\varphi \notin D_{\mathcal{K}}$, that is, $\pi_H  (\varphi) \notin  H_{0} \left( \R^2 \right)$, then the laws of $B_{t}$ and  $B_{t} - \pi_H (\varphi) \left( t \right)$ are mutually singular measures on $\WR$, and hence for all $\varepsilon$ sufficiently small we have that
\begin{align*}
\Prob \left( \max_{0 \leqslant t \leqslant 1} \vert B_{t} - \pi_H(\varphi) \left( t \right) \vert < \varepsilon \right) =0,
\end{align*}
since  $\Prob \left( \max_{0 \leqslant t \leqslant 1} \vert B_{t}  \vert < \varepsilon \right) >0$. Now, note that for all $\varepsilon$ sufficiently small
\begin{align*}
\Prob \left( \max_{0 \leqslant t \leqslant 1} \vert  \mathcal{K} \left( \varphi^{-1} g \right)_{t} \vert < \varepsilon \right)  \leqslant \Prob \left( \max_{0 \leqslant t \leqslant 1} \vert B_{t} - \pi_H(\varphi) \left( t \right) \vert < \varepsilon \right) =0,
\end{align*}
proving that $\mathcal{L} (\varphi) = - \infty$ for $\varphi \notin D_{\mathcal{K}}$. Thus, we showed that $\mathcal{L} (\varphi) $  is finite if and only if $\varphi \in D_{\mathcal{K}}$, that is, $D_{\mathcal{L}} = D_{ \mathcal{K}}$.
\end{proof}

\begin{acknowledgement}
The authors would like to thank Nathaniel Eldredge for pointing out a gap in the previous version of the paper, and Lorenzo Dello Schiavo for useful comments. The second author would like to thank Ofer Zeitouni for suggesting this problem during his visit to the University of Connecticut.
\end{acknowledgement}

\providecommand{\bysame}{\leavevmode\hbox to3em{\hrulefill}\thinspace}
\providecommand{\MR}{\relax\ifhmode\unskip\space\fi MR }
\providecommand{\MRhref}[2]{%
  \href{http://www.ams.org/mathscinet-getitem?mr=#1}{#2}
}
\providecommand{\href}[2]{#2}

\end{document}